\documentclass[a4paper]{article}

\usepackage{amsfonts}
\usepackage{graphicx}
\usepackage{amsthm}
\usepackage{amsmath}
\usepackage{amssymb}
\usepackage{enumerate}
\allowdisplaybreaks

\newtheorem{proposition}{Proposition}[section]
\newtheorem{lemma}[proposition]{Lemma}

\newtheorem{theorem}[proposition]{Theorem}
\newtheorem{definition}[proposition]{Definition}
\newtheorem{conjecture}[proposition]{Conjecture}
\theoremstyle{definition}

\newtheorem{remark}[proposition]{Remark}
\numberwithin{equation}{section}

\begin{document}

\begin{center}
\LARGE
\textbf{Towards a CFSG-free diameter bound for $\mathrm{Alt}(n)$}
\bigskip\bigskip

\large
Daniele Dona\footnote{The author was partially supported by the European Research Council under Programme H2020-EU.1.1., ERC Grant ID: 648329 (codename GRANT).}
\bigskip

\normalsize
Mathematisches Institut, Georg-August-Universit\"at G\"ottingen

Bunsenstra\ss e 3-5, 37073 G\"ottingen, Germany

\texttt{daniele.dona@mathematik.uni-goettingen.de}
\bigskip\bigskip\bigskip
\end{center}

\begin{minipage}{110mm}
\small
\textbf{Abstract.} Helfgott and Seress \cite{HS14} have proved the existence of a quasipolynomial upper bound on the diameter of $\mathrm{Alt}(n)$.

\ \ \ In this paper, we walk partway towards removing the dependence on CFSG from that result, by using the algorithm solving the string isomorphism problem (due to Babai \cite{Bab16a}) in its CFSG-free version (due to Babai \cite[\S 13.1]{Bab16a} and Pyber \cite{Pyb16}): the result contained in here relies on the analysis of Babai's algorithm contained in \cite{Don18b}, based in turn on \cite{Hel19}. Conditional on a conjecture about certain products of small-indexed subgroups (Conjecture~\ref{cj:smallindex}), we provide a CFSG-free proof of a bound on the diameter of $\mathrm{Alt}(n)$ that is better than the already existing CFSG-free results in the literature. In fact, the same bound holds for all transitive permutation subgroups $G\leq\mathrm{Sym}(n)$.

\ \ \ The paper is part of the author's doctoral thesis \cite{Don20}.
\medskip

\medskip

\end{minipage}
\bigskip

\section{Introduction}\label{se:intro}

Babai's conjecture \cite{BS88} is an important open problem in the context of finite group theory. Because of the Classification of Finite Simple Groups (CFSG, see \cite[\S 1.2]{Wil09}), we know that we need only to treat the two cases of groups of Lie type and of alternating groups, and in fact most proofs to date produce results in only one of the two classes.

The strongest known result for the alternating case is 
\begin{equation}\label{eq:hs14}
\mathrm{diam}(\mathrm{Sym}(n)),\mathrm{diam}(\mathrm{Alt}(n))\leq e^{O(\log^{4}n\log\log n)},
\end{equation}
which was proved by Helfgott and Seress \cite{HS14}: it was a big improvement over the previous best bound $e^{(1+o(1))\sqrt{n\log n}}$, due to Babai and Seress \cite{BS88}, and is quite close to the actual order of magnitude that Babai's conjecture anticipates. Later, Helfgott \cite{Hel18} gave a proof of slightly less tight bound for \eqref{eq:hs14} (with $(\log\log n)^{2}$ in the exponent instead of $\log\log n$, see \eqref{eq:hel18}) that made use of a weakened product theorem, so as to provide a more general framework for the problem and shrink the distance between proofs for permutation subgroups and proofs for groups of Lie type: product theorems are tools associated chiefly with the Lie type case (starting with the key proposition of \cite{Hel08}), and such strong results cannot hold for $\mathrm{Alt}(n)$ (see the counterexamples in \cite[\S 4]{Spi12} and \cite[Thm.~17]{PPSS12}), but the weakened product theorem in \cite[Thm.~1.4]{Hel18} shows that drawing a bridge in this direction is still possible.

Both \cite{HS14} and \cite{Hel18} rely in some way on CFSG: as a matter of fact they are both based at their core on the classification of primitive permutation subgroups in primis due to Cameron \cite{Cam81} and refined later by Liebeck \cite{Lie84} and Mar\'oti \cite{Mar02}, which descends from CFSG and the O'Nan-Scott theorem \cite{Sco80}. Removing the dependence on CFSG from \eqref{eq:hs14} or analogous results would be in the words of Helfgott ``another worthwhile goal'' \cite[\S 1]{Hel18}.

Our aim here will be to walk at least part of the way towards that goal: modulo an unproven assumption (Conjecture~\ref{cj:smallindex}), we will be giving a CFSG-free proof of a diameter bound for $\mathrm{Alt}(n)$, and in fact for all transitive subgroups of $\mathrm{Sym}(n)$, which is not as strong as the one given by Helfgott and Seress but is a decisive improvement on the CFSG-free bound due to Babai and Seress; the main result is Theorem~\ref{th:cfsgfree}. To do so, we are going to make use of a tool that has not been adopted before in the context of Babai's conjecture, namely Babai's quasipolynomial algorithm for the graph and string isomorphism problems (GIP and SIP, see \cite{Bab16a} \cite{Hel19} \cite{HBD17}). In particular, the analysis performed in \cite{Don18b} will be instrumental in accomplishing what we want: in brief, walking through the algorithm in the case of a trivial string, one can give a description of the input permutation group that shares some characteristics of Cameron's result even when CFSG is not available. We discuss the strategy behind the present paper in \S\ref{se:cfsgfreeharald}, together with overviewing what is needed from \cite{Hel18} to understand the context of the already known result that relies on CFSG.

\section{Background and strategy}\label{se:cfsgfreeharald}

Helfgott's result \cite[Thm.~6.1]{Hel18} on the diameter of $\mathrm{Alt}(n)$ is the following: if $G=\mathrm{Alt}(n),\mathrm{Sym}(n)$ then
\begin{equation}\label{eq:hel18}
\mathrm{diam}(G)\leq e^{O(\log^{4}n(\log\log n)^{2})}.
\end{equation}
To prove the bound above, he shows that a sort of product theorem also holds in the context of permutation groups. We have already mentioned that product theorems, like the key proposition in \cite{Hel08}, are central to proofs in the Lie type case. Here is the statement.

\begin{proposition}\label{pr:prodtheohel}
Let $p$ be a prime, let $G=\mathrm{SL}_{2}(\mathbb{F}_{p}),\mathrm{PSL}_{2}(\mathbb{F}_{p})$, and let $A$ be a set of generators of $G$. Then there exist absolute constants $\delta>0$ and $k\geq 1$ such that at least one of the following alternatives holds:
\begin{enumerate}[(a)]
\item\label{pr:prodtheohelgrow} $|A^{3}|\geq|A|^{1+\delta}$;
\item\label{pr:prodtheohelfill} $(A\cup A^{-1}\cup\{e\})^{k}=G$.
\end{enumerate}
\end{proposition}

A statement as strong as this cannot hold for $\mathrm{Alt}(n)$: there are specific counterexamples in \cite[\S 4]{Spi12} and \cite[Thm.~17]{PPSS12}. However, the weakened version below is still true.

\begin{theorem}[\cite{Hel18}, Theorem 1.4]\label{th:hel18prod}
Let $G\leq\mathrm{Sym}(n)$ be $3$-transitive, and let $A$ be a set of generators of $G$ with $e\in A=A^{-1}$. Then there are absolute constants $C,k>0$ such that, if $|A|\geq n^{C\log^{2}n}$, then at least one of the following alternatives holds:
\begin{enumerate}[(a)]
\item\label{th:hel18prodgrow} $|A^{n^{C}}|\geq|A|^{1+\frac{k\log\log|A|}{\log^{2}n\log\log n}}$;
\item\label{th:hel18prodfill} there is a transitive $G'\leq\mathrm{Sym}(n')$ with $n'\leq n$ such that $\mathrm{diam}(\mathrm{Cay}(G,A))\leq n^{C}\mathrm{diam}(G')$, and either $n'\leq e^{-\frac{1}{10}}n$ or $G'\neq\mathrm{Sym}(n'),\mathrm{Alt}(n')$.
\end{enumerate}
\end{theorem}

In the above, \textit{$k$-transitive} means that every $k$-tuple of distinct elements can be sent to any other by an element of $G$.

Theorem~\ref{th:hel18prod} still qualifies as a sort of product theorem, in the sense that after as many instances as possible of growth of $|A|$ in case \eqref{th:hel18prodgrow}, like in Proposition~\ref{pr:prodtheohel}\eqref{pr:prodtheohelgrow}, we reach in case \eqref{th:hel18prodfill} a bound on the diameter of the Cayley graph of $G$ and the final power of $A$, which was neater for Proposition~\ref{pr:prodtheohel}\eqref{pr:prodtheohelfill} (it was $3$) whereas now it sparks a recursion process. In this sense, Theorem~\ref{th:hel18prod} is part of an effort to close the gap between the Lie type proofs and the alternating proofs.

What is important for us, though, is that Theorem~\ref{th:hel18prod} implies \eqref{eq:hel18} via
\begin{equation}\label{eq:realbs92}
\mathrm{diam}(G)=e^{O(\log^{2}n)}\prod_{i}\mathrm{diam}(\mathrm{Alt}(m_{i}))
\end{equation}
(which is \cite[Prop.~4.15]{Hel18}, and is part of the aforementioned recursion process), and that the part of the proof of \eqref{eq:hel18} that depends on CFSG is contained solely in \eqref{eq:realbs92}, whereas Theorem~\ref{th:hel18prod} itself is CFSG-free. Therefore, what we need is to show something resembling \eqref{eq:realbs92} without the help of CFSG, and then we can conclude in a way that is not different from what has already appeared in \cite{Hel18}; the end of the proof of Theorem~\ref{th:cfsgfree} will proceed exactly along these lines.

As for the strategy leading to that point, it is as follows. An intermediate result in Helfgott's proof, i.e.\ \cite[Prop.~4.6]{Hel18}, produces a nicely shaped chain of normal subgroups necessary to reach the conclusion \eqref{eq:realbs92}. Let us write it down for future reference.

\begin{proposition}\label{pr:hel1846}
Let $G\leq\mathrm{Sym}(n)$ be transitive. Then there is a composition series $\{e\}=H_{0}\lhd H_{1}\lhd\ldots\lhd H_{\ell}=G$ and a partition $\{C_{1},C_{2}\}$ of the set of composition factors $H_{i}/H_{i-1}$ with the following properties:
\begin{enumerate}[(a)]
\item\label{pr:hel1846smallone} if $H_{i}/H_{i-1}\in C_{1}$ then $H_{i}/H_{i-1}\simeq M_{i}^{k_{i}}$ for some $M_{i}$ simple and $k_{i}\leq 2n$;
\item\label{pr:hel1846smallprod} $\prod_{H_{i}/H_{i-1}\in C_{1}}|H_{i}/H_{i-1}|=n^{O(\log n)}$;
\item\label{pr:hel1846altone} if $H_{i}/H_{i-1}\in C_{2}$ then $H_{i}/H_{i-1}\simeq\mathrm{Alt}(m_{i})^{k_{i}}$ for some $m_{i}\geq 5$ and $k_{i}\leq 2n$;
\item\label{pr:hel1846altprod} $\prod_{H_{i}/H_{i-1}\in C_{2}}m_{i}\leq n$, and each $m_{i}\leq\frac{n}{2}$ unless $G$ is a giant;
\item\label{pr:hel1846length} $\ell=O(\log n)$.
\end{enumerate}
\end{proposition}

We aim to replace this intermediate result with a CFSG-free analogue that would prove a counterpart of \eqref{eq:realbs92}. Cameron's structure theorem is the backbone of the proof of \cite[Prop.~4.6]{Hel18}, as it breaks down permutation groups into pieces that are either small or alternating (represented here by the factors in $C_{1}$ and $C_{2}$ respectively) and allows us to construct the chain; here it is, in its later version by Mar\'oti \cite{Mar02}.

\begin{theorem}\label{th:maroti}
Let $n\geq 1$ and let $G\leq\mathrm{Sym}(n)$ be primitive. Then, one of the following alternatives holds:
\begin{enumerate}[(a)]
\item\label{th:marotialt} there are integers $m,r,k$ such that $\mathrm{Alt}(m)^{r}\leq G\leq\mathrm{Sym}(m)\wr\mathrm{Sym}(r)$, where $\mathrm{Alt}(m)$ acts on $k$-subsets of $\{1,2,\ldots,m\}$ and the wreath product action is the primitive one (so that in particular $n=\binom{m}{k}^{r}$);
\item\label{th:marotimath} $G$ is one of the sporadic groups $\mathrm{M}_{11},\mathrm{M}_{12},\mathrm{M}_{23},\mathrm{M}_{24}$ with their $4$-transitive action;
\item\label{th:marotismall} $|G|\leq n\prod_{i=0}^{\lfloor\log_{2}n\rfloor-1}(n-2^{i})<n^{1+\log_{2}n}$.
\end{enumerate}
\end{theorem}

The ``small'' cases would be the primitive quotients that fall into cases \eqref{th:marotimath} and \eqref{th:marotismall}, while the ``alternating'' case is obtained from the wreath product of case \eqref{th:marotialt}; for the wreath product and its primitive action, see \cite[\S 2.7]{DM96}, and for the Mathieu groups in \eqref{th:marotimath} see \cite[\S 5.2-5.3]{Wil09}. A CFSG-free structure theorem that albeit weaker is still capable of breaking down permutation groups into small and alternating pieces would be a good candidate for being the backbone of our own result. We find such a candidate in \cite[Thm.~3.1]{Don18b}, which is based on Babai's algorithm for SIP both in its CFSG and in its CFSG-free version.

Let us discuss the salient points of the structure of the algorithm. Babai \cite{Bab16a} has produced an algorithm that describes in quasipolynomial time the set $\mathrm{Iso}_{G}(\mathbf{x},\mathbf{y})$ of all permutations in $G$ that send the string $\mathbf{x}$ to the string $\mathbf{y}$: this algorithm is dependent on CFSG, in that it uses Cameron as a crossroad to pass from the original problem to a collection of subproblems with a smaller or more structured $G$. A slightly modified CFSG-free version of the same algorithm has been produced as well, a work started by Babai \cite[\S 13.1]{Bab16a} and concluded by Pyber \cite{Pyb16}: this new version avoids the use of Cameron, but broadly speaking retains the same idea of a ``crossroad through structure theorem'' using a result by Pyber \cite[Thm.~3.15]{Pyb93} and the Split-or-Johnson routine of the original algorithm \cite[\S 7]{Bab16a}.

It is possible to take a general permutation group $G$ and make it pass through Babai's algorithm: after all, $\mathrm{Iso}_{G}(\mathbf{x},\mathbf{y})$ is none other than $G\cap H\sigma$ where $H$ is a product of symmetric groups (one for each distinct letter of $\mathbf{x}$) and $\sigma$ is any one permutation sending $\mathbf{x}$ to $\mathbf{y}$; this means that we can choose $\mathbf{x}_{0}$ to be a constant string (equivalently, choose $H$ to be the whole symmetric group) and we can obtain $G=\mathrm{Aut}_{G}(\mathbf{x}_{0})=\mathrm{Iso}_{G}(\mathbf{x}_{0},\mathbf{x}_{0})$ as a result. Of course from the standpoint of the string isomorphism algorithm this process is utterly useless, since $\mathbf{x}_{0}$ is trivial and the algorithm outputs $G$ having been given $G$ as input; nevertheless, the algorithm is still making $G$ pass through the whole process of reducing it into smaller subgroups, identifying alternating factor, etc...: this is exactly what we want, i.e.\ finding structure inside $G$, and the modifications by Babai and Pyber allow us to do precisely that without resorting to CFSG.

The key observation is that Babai's algorithm takes only quasipolynomial time in $n$, which implies that the information that we retrieve about the structure of the group is also simple enough. For example, the number of floors of the structure tree with which we are going to replace the chain in \cite[Prop.~4.6]{Hel18} (see Proposition~\ref{pr:treegroup}) will be polylogarithmic too.

\begin{remark}\label{re:trivstring}
In using Babai's algorithm to determine the structure of a group $G\leq\mathrm{Sym}(n)$, i.e.\ determine $G=\mathrm{Iso}_{G}(\mathbf{x}_{0},\mathbf{x}_{0})$ with $\mathbf{x}_{0}=\alpha^{n}$, we always reduce to subproblems that also involve only strings of the form $\mathbf{x}'_{0}=\alpha^{n'}$. In fact, the only manipulations of the strings themselves that occur in the algorithm in \cite[\S 6]{Don18b} are restrictions $\mathbf{x}\mapsto\mathbf{x}|_{\Omega}$ and preimages $\mathbf{x}\mapsto\mathbf{x}^{\sigma^{-1}}$, both of which do not change the property of being a constant string. Hence, all subproblems descending from the original problem on $G$ are also problems on some $G'$, and not on a more general coset $G'\cap H'\sigma'$ with $H'\leq\mathrm{Sym}(n')$; in other words, in the language and notation of \cite[Thm.~3.1]{Don18b}, since the first $H$ is $\mathrm{Sym}(n)$, all intermediate $H'$ are $\mathrm{Sym}(n')$ and the final atoms themselves are $\mathrm{Alt}(n')$.

In truth, this does not mean that we never use nonconstant strings in Babai's algorithm, even when starting with $\mathbf{x}_{0}$ constant: some routines feature auxiliary strings made of different letters, such as the ``glauque'' letter in \cite[\S 6.1.2]{Hel19}, but they are used only to gather structural information and the actual $\mathbf{x}_{0}$ does not reduce to them.
\end{remark}

There are of course some important disadvantages in adopting this new path towards a reduction like in \eqref{eq:realbs92}: they are going to be due mainly to the fact that the subgroups involved in the descent process are not necessarily normal, as they were in the situation where Cameron's theorem was a viable route. We will discuss them later in more depth; for now we limit ourselves to observe that the fact that Theorem~\ref{th:cfsgfree} has a weaker final bound and depends on Conjecture~\ref{cj:smallindex} is exactly what we have to pay for this weakening of the intermediate result.

\section{Tools}\label{se:cfsgfreelem}

Let us start with a couple of easy lemmas, describing the structure of $\mathrm{Alt}(n)$.

\begin{lemma}\label{le:alt3cyc}
The group $\mathrm{Alt}(n)$ is generated by the set of $3$-cycles.
\end{lemma}

\begin{proof}
This is elementary. Any element of $\mathrm{Alt}(n)$ is the product of an even number of transpositions $\tau_{i}$, or equivalently a product of $\tau_{2i-1}\tau_{2i}=(a\,b)(c\,d)$: if $b=c$ then $\tau_{2i-1}\tau_{2i}=(a\,d\,b)$ is already a $3$-cycle, and if $b\neq c$ then $\tau_{2i-1}\tau_{2i}=(a\,c\,b)(b\,d\,c)$ is the product of two $3$-cycles.
\end{proof}

\begin{lemma}\label{le:altindex}
For any $n\geq 5$, any proper subgroup of $\mathrm{Alt}(n)$ has index $\geq n$.
\end{lemma}

\begin{proof}
This is a standard result that uses the fact that $\mathrm{Alt}(n)$ is simple for all $n\geq 5$ (see for instance \cite[\S 4.6, Ex.~1]{DF03}). A whole classification of maximal permutation subgroups exists, the O'Nan-Scott theorem \cite{Sco80} already mentioned in \S\ref{se:intro}, but we do not need such a powerful tool here.

Let $G<\mathrm{Alt}(n)$: in particular, $\mathrm{Alt}(n)$ acts by permuting the cosets of $G$ (left cosets, say), so that there is a natural group homomorphism
\begin{equation*}
\varphi:\mathrm{Alt}(n)\rightarrow\mathrm{Sym}([\mathrm{Alt}(n):G]).
\end{equation*}
Since $\mathrm{Alt}(n)$ is simple, the normal subgroup $\mathrm{ker}(\varphi)$ is either $\{e\}$ or $\mathrm{Alt}(n)$; however, there exists an element $\sigma\in\mathrm{Alt}(n)\setminus G$, and then $\sigma$ induces a nontrivial partition of the cosets of $G$, so that $\mathrm{ker}(\varphi)\neq\mathrm{Alt}(n)$. Hence, $\varphi$ is injective, and since we have $(n-1)!=\frac{1}{n-1}n!<\frac{1}{2}n!$ we can conclude that $[\mathrm{Alt}(n):G]\geq n$.
\end{proof}

Thanks to the previous lemmas, we can show the following result, which will prevent the arising of large alternating factors when $G$ itself is not giant (i.e.\ not equal to $\mathrm{Sym}(n)$ or $\mathrm{Alt}(n)$). We also adopt the notations $G_{A},G_{(A)}$ for setwise and pointwise stabilizers respectively, and $G|_{A}$ for the restriction of the group $G$ to $A$ (when it is possible to do so, namely when $G$ already stabilizes $A$).

\begin{proposition}\label{pr:bigalt}
Let $G\leq\mathrm{Sym}(n)$ be a transitive permutation subgroup, with $n\geq 5$. Consider a set $A\subseteq[n]$ with $|A|=\alpha n$ for some $\frac{2}{3}\leq\alpha<1$, and let $H\leq G_{A}$. Suppose that $H|_{A}=\mathrm{Alt}(A)$. Then $G\geq\mathrm{Alt}(n)$.
\end{proposition}

This is the kind of proposition that likely can be proved in several different fashions. If we were allowed to use CFSG for example, we could argue that $G$ must be not only transitive but primitive, because an alternating group inside of it permuting more than half of the vertices prevents the formation of a nontrivial block system, and then we could use Cameron's theorem to exclude the possibility of $G$ not being a giant given that by hypothesis $|G|\geq\frac{1}{2}(\alpha n)!$. For our purposes, however, we will need to provide a proof that does not rely on CFSG.

We remark that there is no particular reason to use $\frac{2}{3}$ as a lower bound for $\alpha$: as one can readily check, we can prove the same for any constant arbitrarily close to $\frac{1}{2}$, as long as we choose $n$ to be large enough.

\begin{proof}
By hypothesis we have that $H|_{A}=\mathrm{Alt}(A)$; the main idea is to prove that $H_{(\bar{A})}|_{A}=\mathrm{Alt}(A)$ as well, where $\bar{A}=[n]\setminus A$.

Consider an arbitrary $x\in\bar{A}$. By the isomorphism theorems we first have that $[\mathrm{Alt}(A):H_{x}|_{A}]\leq[H:H_{x}]$, and in turn we also have that $[H:H_{x}]=[H|_{\bar{A}}:H_{x}|_{\bar{A}}]$ following the same reasoning and using moreover the fact that $H_{x}$ contains the kernel of the restriction map to $\bar{A}$. The subgroup $H_{x}|_{\bar{A}}$ cannot have more than $|\bar{A}|$ cosets inside $H|_{\bar{A}}$ (one can see this as an instance of the orbit-stabilizer theorem); hence
\begin{equation}\label{eq:stabalt}
[\mathrm{Alt}(A):H_{x}|_{A}]\leq(1-\alpha)n<\alpha n=|A|,
\end{equation}
and by Lemma~\ref{le:altindex} we must have $H_{x}|_{A}=\mathrm{Alt}(A)$. Now we can redefine $H$ to be $H_{x}$ acting on $n\setminus\{x\}$, and we can repeat the whole process with a new $x'$: notice that $\alpha$ increases, so that the second inequality inside \eqref{eq:stabalt} is still valid. Iterating the process for all points of the original $\bar{A}$, we obtain in the end $H_{(\bar{A})}|_{A}=\mathrm{Alt}(A)$.

At this point it is easy to conclude. In fact, $H$ (and therefore $G$) contains all the $3$-cycles $(a\,b\,c)$ formed by elements $a,b,c\in A$, so we just have to use them to get all the $3$-cycles in $[n]$ and we could conclude by Lemma~\ref{le:alt3cyc}. Take any $x\in\bar{A}$: since $G$ is transitive there exists a $g\in G$ that sends $x$ to a given element $y\in A$, and since $\alpha\geq\frac{2}{3}$ and $n\geq 5$ there exist two elements $r,s\in A\setminus g(\bar{A})$; then $g(r\,s\,y)g^{-1}$ is the $3$-cycle $(g^{-1}(r)\,g^{-1}(s)\,x)$, which contains two elements of $A$ and one element of $\bar{A}$. Using
\begin{align*}
(a\,b\,c) & =(b\,c\,a)=(c\,a\,b), \\
(a\,c\,b) & =(a\,b\,c)^{2}, \\
(b\,c\,x) & =(a\,c\,b)(a\,b\,x)(c\,a\,b),
\end{align*}
we can then reorder elements and insert elements from other cycles as we please, and get all the $3$-cycles of $[n]$.
\end{proof}

For any two groups $H\leq G$, let us denote by $\mathcal{L}(G,H)$ the set of left cosets of $H$ inside $G$: to prevent confusion we would rather avoid using the notation $G/H$ for such a set, unless we are dealing with a normal subgroup $H$ and $G/H$ is the quotient group\footnote{The author is embarrassingly prone to get confused by the notation and assume that $H$ is normal whenever $G/H$ is written on paper. May the reader be indulgent with him.}. We are going to work with a class of Schreier graphs arising from the action on the cosets of a subgroup; incidentally, this was the context in which Schreier graphs were originally conceived \cite{Sch27}.

\begin{definition}\label{de:diamcoset}
Let $G$ be a group and let $H\leq G$. We define $\mathrm{diam}(G,H)$, the diameter of the pair $(G,H)$, to be the maximum among the (undirected) diameters of all the Schreier graphs $\mathrm{Sch}(\mathcal{L}(G,H),S)$, where $S$ runs through all sets of generators of $G$ and the action $\eta:G\times\mathcal{L}(G,H)\rightarrow\mathcal{L}(G,H)$ defining the graphs is the left multiplication $\eta(g,g'H)=gg'H$.
\end{definition}

The diameter of a group $\mathrm{diam}(G)$ is then the same as $\mathrm{diam}(G,\{e\})$, and if $H$ is normal in $G$ then $\mathrm{diam}(G/H)=\mathrm{diam}(G,H)$. Of course, there is nothing special about our choice of ``left'': we could as well define $\mathcal{R}(G,H)$, and act on it through right multiplication.

We use here Schreier's lemma so as to be able to use a chain of subgroups as a way to bound diameters. The use we make of it is identical to what happens with \cite[Lemma~4.7]{Hel18}.

\begin{lemma}\label{le:schreierindex}
Let $G$ be a finite group, let $H\leq G$ be a proper nontrivial subgroup, and let $S$ be a set of generators of $G$ with $e\in S=S^{-1}$. Then
\begin{align*}
\mathrm{diam}(G) & \leq(2\mathrm{diam}(G,H)+1)\mathrm{diam}(H)+\mathrm{diam}(G,H) \\
 & \leq 4\mathrm{diam}(G,H)\mathrm{diam}(H).
\end{align*}
\end{lemma}

\begin{proof}
First we prove the following claim: if $d=\mathrm{diam}(G,H)$ then $S^{2d+1}\cap H$ generates $H$. Calling $\pi:G\rightarrow\mathcal{L}(G,H)$ the natural projection, by definition we have $\pi(S)^{d}=\mathcal{L}(G,H)$; this equality means that $S^{d}$ contains at least one representative for each coset $gH$ in $G$. For any coset $gH$, choose a representative $\tau(g)\in S^{d}$. Then, for any $h\in H$ and any way to write $h$ as a product of elements $s_{i}\in S$, we have
\begin{align*}
h= & \ s_{1}s_{2}\ldots s_{k}= \\
= & \ (s_{1}\tau(s_{1})^{-1})\cdot(\tau(s_{1})s_{2}\tau(\tau(s_{1})s_{2})^{-1})\cdot\ldots\cdot(\tau(\tau(\tau(\ldots)s_{k-2})s_{k-1})s_{k}).
\end{align*}
Each element of the form $\tau(x)s_{i}\tau(\tau(x)s_{i})^{-1}$ is contained in $S^{2d+1}\cap H$, so the same can be said about the last element of the form $\tau(x)s_{k}$ (since $h$ itself is in $H$); therefore $S^{2d+1}\cap H$ is a generating set of $H$.

The result is now easy: if $S^{2d+1}\cap H$ generates $H$, then $(S^{2d+1})^{\mathrm{diam}(H)}\supseteq H$ and since $S^{d}$ contains by definition representatives of all the left cosets of $H$ inside $G$ we have $S^{d}H=G$, thus concluding the proof.
\end{proof}

The condition of $H$ being proper nontrivial is really only needed for the second inequality, since by definition $\mathrm{diam}(\{e\})=0$. For ease of notation, we can use the second inequality anyway and conventionally establish that $\mathrm{diam}(\{e\})=1$ (which we are going to do).

\section{Main theorem}\label{se:cfsgfreemain}

Now we begin our path towards the main result (Theorem~\ref{th:cfsgfree}). First, let us rewrite \cite[Thm.~3.1]{Don18b} in a form that suits us more.

\begin{proposition}\label{pr:treegroup}
Let $n\geq 1$ and let $G_{0}\leq\mathrm{Sym}(n)$ acting on a set $\Omega_{0}$ of size $n$. Then we can build a rooted tree $T(G_{0},\Omega_{0})$ (oriented away from the root, say) with the following properties:
\begin{enumerate}[(a)]
\item\label{pr:treegroupcol} the vertices are pairs $(G,\Omega)$ and the edges are coloured either ``$\mathrm{(\mathcal{C}1)}$'', ``$\mathrm{(\mathcal{C}2)}$'' or ``$\mathrm{(\mathcal{C}3)}$'';
\item\label{pr:treegroupextr} the root is $(G_{0},\Omega_{0})$ and the leaves are $(\mathrm{Alt}(\Omega_{i}),\Omega_{i})$ for a partition $\{\Omega_{i}\}_{i}$ of $\Omega_{0}$;
\item\label{pr:treegroupmid} for any non-leaf vertex $(G,\Omega)$, either:
\begin{enumerate}[(1)]
\item\label{pr:treegroupmidc1} there is only one edge departing from it, coloured ``$\mathrm{(\mathcal{C}1)}$'', and its endpoint is $(G',\Omega)$ for some $G'\leq G$, or
\item\label{pr:treegroupmidc2} there are only edges coloured ``$\mathrm{(\mathcal{C}2)}$'' departing from it, and their endpoints are $(G|_{\Omega_{i}},\Omega_{i})$ for some nontrivial partition $\{\Omega_{i}\}_{i}$ of $\Omega$, or
\item\label{pr:treegroupmidc3} there is only one edge departing from it, coloured ``$\mathrm{(\mathcal{C}3)}$'', and its endpoint is $(G',\Omega)$ for some $G'\lhd G$ with $G/G'$ isomorphic to an alternating group of degree $\geq 5$;
\end{enumerate}
\item\label{pr:treegroupc3c2} if a vertex $(G',\Omega)$ has an incoming edge coloured ``$\mathrm{(\mathcal{C}3)}$'' coming from a vertex $(G,\Omega)$, then it has departing edges coloured ``$\mathrm{(\mathcal{C}2)}$'' whose endpoints $(G'|_{\Omega_{i}},\Omega_{i})$ are such that $|\Omega|\geq m|\Omega_{i}|$ for all $i$, where $G/G'\simeq\mathrm{Alt}(m)$;
\item\label{pr:treegroupboundc1} every index $[G:G']$ coming from an edge $((G,\Omega),(G',\Omega))$ coloured ``$\mathrm{(\mathcal{C}1)}$'' is bounded by $n^{O(\log^{5}n)}$, and for any path from the root to a leaf the number of edges coloured ``$\mathrm{(\mathcal{C}1)}$'' lying on the path is bounded by $O(\log^{2}n)$;
\item\label{pr:treegroupboundc3} for any path from the root to a leaf, the product $\prod_{i}m_{i}$ of the degrees of the alternating groups coming from all the edges coloured ``$\mathrm{(\mathcal{C}3)}$'' lying on the path and from the final leaf is bounded by $n$.
\end{enumerate}
\end{proposition}

If we were to compare the tree above with the chain in Proposition~\ref{pr:hel1846}, the two parts $C_{1},C_{2}$ would correspond here to ($\mathcal{C}1$) and ($\mathcal{C}3$) respectively, and the bounds in Proposition~\ref{pr:hel1846}\eqref{pr:hel1846smallprod}-\eqref{pr:hel1846altprod}-\eqref{pr:hel1846length} would correspond to those in Proposition~\ref{pr:treegroup}\eqref{pr:treegroupboundc1}-\eqref{pr:treegroupboundc3}. The normality of the subgroups involved adds the following perk: all composition factors fit into one chain, by making them into direct product of simple groups; instead, in Proposition~\ref{pr:treegroup} we are forced to deal with a tree, with bifurcations labelled ($\mathcal{C}2$).

\begin{proof}
The construction of $T(G_{0},\Omega_{0})$ comes as we said from the use of \cite[Thm.~3.1]{Don18b} in the case of $H=\mathrm{Sym}(n)$. Its definition is similar to that, widely used, of a \textit{structure tree} as in \cite[\S 4.1]{Hel18} and a \textit{structure forest} in \cite[\S 3]{LM88} \cite[\S 3.4]{BS92}, although it is more refined to suit our needs.

The root is the starting point of the algorithm, i.e.\ the input made of the group $G_{0}$ and the set $\Omega_{0}$ on which the group acts, while the leaves are the atoms that are reached at the end of the procedure; as we said before in Remark~\ref{re:trivstring}, starting with $G_{0}$ makes us reach simple alternating groups instead of the more general possibilities described in ($\mathcal{A}$). The edges leaving a vertex represent the three possibilities ($\mathcal{C}1$)-($\mathcal{C}2$)-($\mathcal{C}3$) in which an expression can break down to smaller expressions as described in the theorem; however, the construction is not exactly like giving to each vertex its smaller expressions as children.

In the case of ($\mathcal{C}1$), in $T(G_{0},\Omega_{0})$ we pass from $G$ to a subgroup $G'$ as its only child. By Remark~\ref{re:trivstring}, the group $H$ in this intermediate step is still $\mathrm{Sym}(\Omega)$, so there is no loss of information: we are simply writing $G=\bigcup_{i}G'\sigma_{i}$ for a set of representatives $\{\sigma_{i}\}_{i}$ of $G'$ in $G$, so that the various subproblems (the smaller well-formed expressions in the language of ($\mathcal{C}$1) inside \cite[\S 3]{Don18b}) are all on the subgroup $G'$.

In the case of ($\mathcal{C}2$), following its exact wording we would reduce from $(G,\Omega)$ to $(\pi_{1}(G),\Omega_{1})$ and $(\pi_{2}(G),\Omega_{2})$ for a partition $\Omega=\Omega_{1}\sqcup\Omega_{2}$ respected by $G$: this is because, as in ($\mathcal{C}1$), $H=\mathrm{Sym}(\Omega)$ by Remark~\ref{re:trivstring}. However, for simplicity we can reduce directly to subdividing $\Omega$ into its orbits\footnote{The order in which we subdivide $\Omega$ is relevant only when starting with nonconstant strings in the original algorithm.}.

The case of ($\mathcal{C}3$) is as described in the theorem: the only child of $G$ is a $G'$ such that $\langle G'\cup\{\sigma_{1},\sigma_{2}\}\rangle=G$ and the group generated by $\sigma_{1},\sigma_{2}$ is some alternating group; let us prove the stronger claims that are present in our statement. The only time ($\mathcal{C}3$) emerges in the CFSG-free algorithm of \cite[\S 6]{Don18b} is in \cite[Prop.~6.16(a)]{Don18b}, where $G$ acts on $\Omega$ preserving a system of blocks $\mathcal{B}$ on which it acts as $\mathrm{Alt}(\Gamma)$ acts on $\binom{\Gamma}{k}$ for some $\Gamma,k$; in general we have some large set $S_{\mathbf{x}}\subseteq\Gamma$, canonical with respect to the string $\mathbf{x}$, such that for any $\sigma\in\mathrm{Alt}(S_{\mathbf{x}})$ there is an element of $\mathrm{Aut}_{G}(\mathbf{x})$ inducing $\sigma$ on $S_{\mathbf{x}}$, and that set would be the origin of our alternating quotient (see \cite[Cor.~6.12]{Don18b}, which traces in more detail the steps we are describing): however for us $\mathbf{x}$ is constant by Remark~\ref{re:trivstring} and $\mathrm{Aut}_{G}(\mathbf{x})=G$, so we can assume $S_{\mathbf{x}}=\Gamma$. Then our $G'$ is the preimage of $\{e\}=\mathrm{Alt}(\Gamma)_{(S_{\mathbf{x}})}$ and our $G$ is the preimage of $\mathrm{Alt}(\Gamma)=\mathrm{Alt}(\Gamma)_{S_{\mathbf{x}}}$ (by definition); hence $G'\lhd G$ and $G/G'\simeq\mathrm{Alt}(\Gamma)$, and since the algorithm passes through ($\mathcal{C}3$) only under the condition $|\Gamma|=m>102\log^{2}n$ we have also $|\Gamma|\geq 5$.

To prove \eqref{pr:treegroupc3c2}, observe that from what we just said in the case of ($\mathcal{C}3$) we have that $G'$ stabilizes the blocks of $\mathcal{B}$ and $G$ permutes them as $\mathrm{Alt}(\Gamma)$ permutes $\binom{\Gamma}{k}$: therefore, since $G'$ is intransitive, the next step will be the restriction to the orbits of the action, i.e.\ ($\mathcal{C}2$), and each new orbit will be of the same size $\binom{m}{k}^{-1}|\Omega|\leq\frac{|\Omega|}{m}$ where $|\Gamma|=m$.

To see \eqref{pr:treegroupboundc1}, let us turn to the proof of \cite[Thm.~3.1]{Don18b}: for each use of ($\mathcal{C}1$), the number of subproblems to which the original problem reduces is bounded as $n^{O(\log^{5}n)}$, as stated in \cite[Prop.~6.15-6.16-6.17]{Don18b}; furthermore the number of subproblems is the same as the index $[G:G']$, since the reduction we are performing each time is as in \cite[Prop.~6.3]{Don18b}. On the other hand, let us examine the four actions we are allowed to do as described in the course of the proof in \cite[\S 7]{Don18b}: the first two involve at most one instance of use of ($\mathcal{C}1$) followed by a reduction through ($\mathcal{C}2$) from $\Omega$ to orbits of size $\leq\frac{2}{3}|\Omega|$; the third involves one instance of ($\mathcal{C}1$) in exchange for a coarser block system in $\Omega$; the fourth involves one ($\mathcal{C}1$) for a reduction of the degree of the smallest symmetric group (that we know of) containing $G$, from $m$ to $1+\sqrt{2m}$. The last two actions can happen at most $O(\log n)$ and $O(\log\log n)$ times respectively on the same $\Omega$, and the first two (which shrink $\Omega$ by a fraction) can happen at most $O(\log n)$ times on a single path of the tree: thus, at most $O(\log^{2}n)$ edges coloured ``$\mathrm{(\mathcal{C}1)}$'' can exist on such a path.

Finally, \eqref{pr:treegroupboundc3} is a consequence of \eqref{pr:treegroupc3c2}: every time we use ($\mathcal{C}3$) with some $\mathrm{Alt}(m)$ associated to it, we are also dividing the orbit size by at least $m$, so that on a path we must have $\prod_{i}m_{i}\leq n$.
\end{proof}

\begin{remark}\label{re:treelang}
A language note: when talking informally about the tree, we will figure the root on top and the paths departing from the root to be vertically descending\footnote{We imagine a genealogical tree, with the ancestral root on top, rather than a real-life tree springing from the ground up. If ancient Berbers had conquered the world, maybe writing conventions and botany would have been in agreement today.}. Thus, expressions like ``descending the tree'' mean for us ``walking along its paths while moving away from the root'', and anything ``horizontal'' is on the contrary something that singles one element out of a path across multiple paths. We also refer to elements (i.e.\ vertices or edges) preceding, following or being between others, or also being closer or farther away than others: all of them refer to their distance from the root of the tree in the usual graph metric.
\end{remark}

As we mentioned in \S\ref{se:cfsgfreeharald}, this new route going through Proposition~\ref{pr:treegroup} has some important disadvantages, descending from this one fact: the reduction process may involve subgroups with small index that are not necessarily normal.

The first consequence of this is our inability to use \cite[Lemma~4.7]{Hel18}, i.e.\ bounding the diameter of $G$ by the product of the diameters of $N,G/N$ (a consequence of Schreier's lemma); on the other hand, the diameter of $G/N$ is trivially bounded by the size of $G/N$ itself, exactly because the small groups are small enough that we do not need anything more clever than that: therefore, we as well do not have any issue in using Schreier's lemma again (Lemma~\ref{le:schreierindex}) and get a multiplication by the index $[G:N]$.

The second, and most dire, consequence is the fact that, as we cannot pass to the normal core of our subgroups (which on the contrary was possible in \cite[Lemma~4.2]{Hel18}), we cannot treat all orbits at the same time and reduce the subgroup tree to a subgroup chain: in this way we are forced to treat all the groups of the tree at once. The alternating groups can indeed be worked with horizontally quite well, thanks to the results on products of simple groups (\cite[Thm.~1.1]{Don19a}, or \cite[Lemma~4.13]{Hel18}). A bound of the form $\prod_{i}m_{i}\leq n$ for the set of degrees $m_{i}$ we need to consider is too strong to be within our reach: by Proposition~\ref{pr:treegroup}\eqref{pr:treegroupboundc3} this holds on a single branch, but it is not sufficient if we are not passing to the normal core; as a consequence, the final bound in Theorem~\ref{th:cfsgfree} is not polylogarithmic in $|G|$ as in \cite{HS14}, but it is still better than any $e^{n^{\varepsilon}}$, and more. The problem of treating the small indices horizontally is in that sense the only difficulty that lies in the way of producing a CFSG-free proof of a diameter bound for transitive groups.

Let us first introduce some notions that will define more clearly what we mean when we talk about a horizontal treatment of the tree.

\begin{definition}\label{de:horcut}
Let $T$ be a tree as in Proposition~\ref{pr:treegroup}.

A horizontal cut of the tree is a set $C$ of vertices and edges of $T$ such that for any path from the root to a leaf there is a unique element of $C$ lying on the path. If a horizontal cut is made only of vertices, we call it a horizontal section.

Two distinct horizontal cuts $C_{1},C_{2}$ are non-crossing if, for every path from the root to a leaf, the vertex or edge of $C_{1}$ lying on the path always precedes or coincides with the vertex or edge of $C_{2}$ (or vice versa). Two horizontal cuts inside a set $S$ of non-crossing cuts are consecutive if there are no other cuts in $S$ lying between them.

A horizontal cut is a $\mathrm{(\mathcal{C}1)}$-cut (respectively $\mathrm{(\mathcal{C}2)}$-cut, $\mathrm{(\mathcal{C}3)}$-cut) if it is not a horizontal section and all its edges are coloured ``$\mathrm{(\mathcal{C}1)}$'' (respectively ``$\mathrm{(\mathcal{C}2)}$'', ``$\mathrm{(\mathcal{C}3)}$'').
\end{definition}

Let us introduce now the result that we are going to use to deal horizontally with the alternating factors.

\begin{proposition}\label{pr:proddiam}
Let $G=\prod_{i=1}^{n}T_{i}$, where $T_{i}=\mathrm{Alt}(m_{i})$ and each $m_{i}\geq 5$. Call $d=\max\{\textup{diam}(T_{i})|1\leq i\leq n\}$ and $m=\max\{m_{i}|1\leq i\leq n\}$. Then
\begin{equation*}
\textup{diam}(G)<\frac{196}{243}n^{3}\cdot 5md.
\end{equation*}
\end{proposition}

\begin{proof}
See \cite[Lemma 4.13]{Hel18} for a proof of the bound $O(n^{3}md)$, and \cite{Don19a} for the explicit constant.
\end{proof}

The result in \cite{Don19a} is more general, as it deals with any $T_{i}$ simple; however, we cannot invoke its proof directly because its case-by-case subdivision and Ore's conjecture (on which it relies, see \cite{Ore51} \cite{LOST10}) depend on CFSG, while Helfgott's version for the alternating group uses the much older result in \cite{Mil99}. The proof of the explicit constant in \cite{Don19a} would work the same way though, even if we only had \cite{Mil99} at our disposal when focusing on the alternating case, so we can insert its constants inside Proposition~\ref{pr:proddiam} as well.

Let us also define precisely what the gap in the argument for small indices is. We do so by formulating the following conjecture (as said after Lemma~\ref{le:schreierindex} we can adopt the convention that $\mathrm{diam}(H')=1$ when $H'=\{e\}$, for ease of notation).

\begin{conjecture}\label{cj:smallindex}
Let $G\leq\mathrm{Sym}(n)$ be a transitive permutation subgroup, let $G_{1},G_{2},\ldots,G_{k}$ be finite groups lying on a horizontal section of the tree built from $G$ as in Proposition~\ref{pr:treegroup}, and let $G'_{i}$ be a subgroup of $G_{i}$ for each $1\leq i\leq k$. Let $H\leq G_{1}\times\ldots\times G_{k}$, and let $H'=H\cap(G'_{i}\times\ldots\times G'_{k})$. Then, there are absolute constants $C_{1},C_{2}>0$ such that
\begin{equation*}
\mathrm{diam}(H)\leq C_{1}k^{C_{2}}\cdot\max\{[G_{i}:G'_{i}]|1\leq i\leq k\}\cdot\mathrm{diam}(H').
\end{equation*}
\end{conjecture}

The dependence of the diameter of a group $G$ on the product between $\mathrm{diam}(H)$ and $\mathrm{diam}(G,H)$ (see Lemma~\ref{le:schreierindex}) on one hand, and the dependence of the product of diameters of simple groups on the maximum of the diameters of the factors (see Proposition~\ref{pr:proddiam} or \cite[Thm.~1.1]{Don19a}) on the other, are the clear influences in the formulation of the conjecture above. The assumption is strong enough to be compatible with a proof of a diameter bound for transitive permutation subgroups that is as strong as in \cite{HS14}; a result like Theorem~\ref{th:cfsgfree}, which provides a qualitatively weaker statement, can be proved even with a weaker version of Conjecture~\ref{cj:smallindex}.

We remark that the condition that the groups $G_{i}$ should be part of the same horizontal section inside the tree cannot be completely dropped. One can choose $G_{i}$ to be the cyclic group generated by a $p_{i}$-cycle, where $p_{i}$ is the $i$-th prime, $G'_{i}$ to be the trivial subgroup, and $H$ to be the whole product: in \cite[Thm.~1.1]{Don19a} we have bounded the diameter of $G_{1}\times\ldots\times G_{k}$ by the product of the primes $p_{i}$, and the bound is tight up to constant; since $p_{k}=(1+o(1))k\log k$ and $\prod_{i=1}^{k}p_{i}=e^{(1+o(1))k\log k}$ by the prime number theorem, a bound like the one in Conjecture~\ref{cj:smallindex} for general $G_{i}$ is false. From another perspective, the conjecture can be seen as limiting the possibilities for groups appearing in horizontal sections across all transitive groups $G$.

Before we move to the main theorem, where we use Conjecture~\ref{cj:smallindex} for our purposes, let us remark that the conjecture itself is true in the case $k=1$, by Lemma~\ref{le:schreierindex} and the trivial bounds $\mathrm{diam}(H,H')\leq[H:H']\leq[G_{1}:G'_{1}]$. In fact, for $k=1$ we can easily prove even more and replace $[G_{1}:G'_{1}]$ by the tighter $\mathrm{diam}(G_{1},G'_{1})$, thanks to the following result (which we also need in the course of the proof of the main theorem anyway).

\begin{proposition}\label{pr:subdiam}
Let $G$ be a finite group and let $H\leq G$; let $G'\leq G$ and $H'=G'\cap H$. Then
\begin{equation*}
\mathrm{diam}(G',H')\leq\mathrm{diam}(G,H).
\end{equation*}
\end{proposition}

\begin{proof}
Let $S'$ be a set of generators of $G'$: we will prove that there is a set $S\supseteq S'$ of generators of $G$ such that the Schreier graph $\mathrm{Sch}(\mathcal{L}(G',H'),S')$ is an induced subgraph of $\mathrm{Sch}(\mathcal{L}(G,H),S)$, so that in particular the diameter of the former is bounded from above by that of the latter.

First of all, we define an appropriate bijection $\varphi$ between the set $\{g'H|g'\in G'\}\subseteq\mathcal{L}(G,H)$ and $\mathcal{L}(G',H')$, simply by $\varphi(g'H)=g'H'$. The map is well-defined: if $g'_{1}H=g'_{2}H$ then ${g'_{1}}^{-1}g'_{2}\in G'\cap H=H'$ and $g'_{1}H'=g'_{2}H'$; it is surjective because if $xH'\in\mathcal{L}(G',H')$ then in particular $xH'\subseteq G'$, which means that $x\in G'$, and it is injective because if $g'_{1}H'=g'_{2}H'$ then ${g'_{1}}^{-1}g'_{2}\in H'\leq H$ and $g'_{1}H=g'_{1}{g'_{1}}^{-1}g'_{2}H=g'_{2}H$. This bijection has also the property of respecting the edges of the graphs we are working with: for any $s'\in S'$ and any $g',g''\in G'$, we have $s'(g'H)=g''H$ if and only if $s'(g'H')=g''H'$ (since $g''^{-1}s'g'\in G'\cap H=H'$); this means that the edges of $\mathrm{Sch}(\mathcal{L}(G,H),S)$ corresponding to elements of $S'$ draw exactly the Schreier graph of $\mathcal{L}(G',H')$ on the vertices of the subset $\{g'H|g'\in G'\}$.

We have just to ensure that we can complete $S'$ to a set of generators $S$ of the whole $G$ without introducing any new edges between the vertices of $\{g'H|g'\in G'\}$. That is however easy to do: it is sufficient to take a finite set $\{s_{1},s_{2},\ldots,s_{k}\}$ of new elements of $G$ that do not belong to $G'$ ensuring only that at every step $s_{i}\not\in\langle G'\cup\{s_{1},\ldots,s_{i-1}\}\rangle$, until we cannot do so anymore. The resulting set $S=S'\cup\{s_{1},s_{2},\ldots,s_{k}\}$ generates $G$, and since $s_{i}\not\in G'$ we have $s_{i}g',s_{i}^{-1}g'\not\in G'$ as well for any $g'\in G'$, so that an edge that starts from or ends into a vertex $g'H$ must have a coset $gH$ with $g\not\in G'$ as its other vertex.
\end{proof}

Now we move to the main theorem.

\begin{theorem}\label{th:cfsgfree}
Assume that we can prove Conjecture~\ref{cj:smallindex} without using CFSG. Let $n$ be large enough. Then, for any transitive permutation subgroup $G\leq\mathrm{Sym}(n)$, we can bound
\begin{equation*}
\mathrm{diam}(G)\leq e^{e^{\frac{1}{\log 2}(\log\log n)^{2}}}
\end{equation*}
without using CFSG.
\end{theorem}

As asserted before, the bound above is worse than the ones reached using CFSG, namely \eqref{eq:hs14} and \eqref{eq:hel18}, but it is a large improvement over the best known bounds that do not use CFSG, which are the already cited $e^{(1+o(1))\sqrt{n\log n}}$ for $G=\mathrm{Sym}(n),\mathrm{Alt}(n)$ \cite{BS88} and
\begin{equation*}
\mathrm{diam}(G)\leq e^{4\sqrt{n}\log^{2}n}
\end{equation*}
for any $G$ primitive not giant, due to Babai \cite[Cor.~1.2]{Bab82}. For comparison, both bounds would correspond to having $\frac{1}{2}\log n+O(\log\log n)$ in the double exponential instead of $\frac{1}{\log 2}(\log\log n)^{2}$; the bound
\begin{equation*}%\label{eq:bt16}
\mathrm{diam}(G)\leq\max\{|G|^{\varepsilon},C_{\varepsilon}\},
\end{equation*}
due to Breuillard and Tointon \cite{BT16} and applying to all $G$ non-abelian simple groups, would have $\log n+\log\log n-O(1)$ (all the advantage of a small $\varepsilon$ would just contribute to the size of the $O(1)$; the result is to be read as ``for every $\varepsilon>0$ there is a $C_{\varepsilon}>0$ such that...''). On the other hand, the known bounds with CFSG would correspond to $(4+o(1))\log\log n$, and Babai's conjecture to $\log\log n+O(1)$.

\begin{proof}
Let us draw the tree $T$ associated with our $G$ as described in Proposition~\ref{pr:treegroup}, and call $\Omega$ the set of size $n$ on which $G$ acts. We are going to artificially lengthen it one step further: from every leaf $(\mathrm{Alt}(\Omega_{i}),\Omega_{i})$, if $|\Omega_{i}|\geq 5$ we add one more ($\mathcal{C}$3) edge to a new vertex ($\{e\},\Omega_{i})$, otherwise the same edge can be labelled as ($\mathcal{C}$1); then, ($\mathcal{C}$2) edges are added to split $\Omega_{i}$ into singletons. Now the leaves of the tree $T$ are all of the form $(\{e\},\{x\})$, and all properties of Proposition~\ref{pr:treegroup} are still respected.

In order to prove our bound, we are going to start from the root $(G,\Omega)$ and descend down the tree one horizontal section at the time, bounding every time the increase in diameter using Lemma~\ref{le:schreierindex}, until we end at the leaves. To get the bound we desire, we will have to be careful in choosing how to descend along the various branches: we need to take advantage of the fact that many contemporaneous descents on multiple branches, either by alternating factors or by small factors, cost as much as only one of them by Proposition~\ref{pr:proddiam} and Conjecture~\ref{cj:smallindex} respectively. To do so, we will define appropriate horizontal cuts to work with.

Let us start with the ($\mathcal{C}$3)-cuts. By Proposition~\ref{pr:treegroup}\eqref{pr:treegroupmid}, all the vertices $(G',\Omega')$ are such that $G'\leq G_{\Omega'}|_{\Omega'}$. In the case of a ($\mathcal{C}$3) edge $((G',\Omega'),(G'',\Omega'))$, the alternating group $G'/G''$ acts on a system of blocks in $\Omega'$ as $\mathrm{Alt}(m)$ acts on some $\binom{m}{k}$, and the blocks themselves are stabilized by $G''$; if $m\geq\frac{2}{3}n$ then the blocks are of size $1$, $G''$ is the trivial subgroup and $k=1$: therefore $G'=\mathrm{Alt}(\Omega')$ and, by Proposition~\ref{pr:bigalt}, $G$ must be a giant. For the following discussion on the tree of $G$, we will assume that our $G\leq\mathrm{Sym}(n)$ is transitive but not a giant, so that we are able to assume that every alternating group associated to a ($\mathcal{C}$3) edge has degree $\leq\frac{2}{3}n$.

We construct a first ($\mathcal{C}$3)-cut $C_{1}$ in the following way. We start with the ($\mathcal{C}$3) edge with the alternating group with the largest degree (or one of them arbitrarily chosen, if more than one exist), and put it in $C_{1}$; then we discard all edges lying on any path from the root to a leaf passing from the edge we have chosen (in other words, all ancestors and descendants), we choose again the ($\mathcal{C}$3) edge with the largest degree among all the remaining ones and we put it in $C_{1}$. We discard the edges lying on a path passing through the second edge we have chosen, and repeat the process until all the edges we have left (if any) are either ($\mathcal{C}$1) or ($\mathcal{C}$2): at this point, we arbitrarily choose vertices on the remaining paths one by one and put them in $C_{1}$, discarding every time all the edges lying on a path through the vertex we choose, until no edges at all are left. By construction, $C_{1}$ is a ($\mathcal{C}$3)-cut.

$C_{1}$ divides the tree $T$ into two parts, the one closer to the root (a tree as well) and the one closer to the leaves (a forest); any vertex belonging to $C_{1}$ is defined to be in both parts, for the sake of simplicity (it will not matter in what follows, since by construction both the edge that precedes such a vertex and all the edges that follow it cannot be ($\mathcal{C}$3)). We repeat the construction of ($\mathcal{C}$3)-cuts as above, in both parts, and obtain two ($\mathcal{C}$3)-cuts $C_{2},C_{3}$. Then we repeat the same construction on the four parts in which we have divided the original tree, and do so $r$ times ($r\geq 1$ to be set later) obtaining in the end ($\mathcal{C}$3)-cuts $C_{1},C_{2},\ldots,C_{2^{r}-1}$: we call these the \textit{thick cuts}.

If there are still some ($\mathcal{C}$3) edges in $T$ that have not been put in any thick cut, we will construct other ($\mathcal{C}$3)-cuts, which we call the \textit{thin cuts}. For any of the $2^{r}$ parts in which $T$ is divided by the thick cuts, we do the following: we take an arbitrary path from a root to a leaf (where the roots are now, quite naturally, the vertices that were the closest to the original root in $T$), choose the first ($\mathcal{C}$3) edge we find and put it in the first ($\mathcal{C}$3)-cut (or we choose an arbitrary vertex, if no such edge exists), discard all the paths passing through our choice, take a second path and repeat until all the paths have been considered or discarded; after creating the first such cut (call it $C$), we discard completely its edges and all edges that precede $C$ in the part of $T$ we are examining, and start again as before with the construction of a second ($\mathcal{C}$3)-cut $C'$. We discard anything that precedes or belongs to $C'$, and repeat until no ($\mathcal{C}$3) edge is left in this part of $T$.

In this way, we have created a set of ($\mathcal{C}$3)-cuts, thick and thin, such that every ($\mathcal{C}$3) edge sits in exactly one of them and such that any two cuts are non-crossing. More interestingly, if $C$ is one of these cuts and $m(C)$ is the maximal degree among the alternating groups of all the ($\mathcal{C}$3) edges of $C$, we can give bounds on $m(C)$ that will be useful to us.

By what we said before, we already have $m(C_{1})\leq\frac{2}{3}n$; for the other cuts we can do better than that. Consider any ($\mathcal{C}$3) edge $e_{2}\in C_{2}$; by construction, there must be a path passing through it that contains a ($\mathcal{C}$3) edge with a degree at least as large as the one of $e_{2}$, and this would be the unique edge $e_{1}$ belonging to both $C_{1}$ and that path: if all paths through $e_{2}$ intersected $C_{1}$ either in edges of smaller degree or in vertices, then $e_{2}$ itself would have belonged to $C_{1}$ in the first place. Hence, Proposition~\ref{pr:treegroup}\eqref{pr:treegroupboundc3} implies that $m(C_{2})\leq\sqrt{n}$ (and $m(C_{3})$ as well).

For any ($\mathcal{C}$3) edge $e_{4}\in C_{4}$, by construction there must be a path with \textit{two} edges with degrees at least as large as its degree. As before, there is a path with the degree of the edge $e_{1}$ lying in $C_{1}$ at least as large, and there is a path (built in the part of $T$ defined by $C_{1}$ in which $e_{4}$ lies) with an edge $e_{2}$ in $C_{2}$ of degree at least as large (say $C_{2}$ is the cut lying in the same part as $e_{4}$, otherwise we say the same for $C_{3}$); we have however to guarantee that the path is the same through both $e_{1}$ and $e_{2}$. If $e_{2}$ is closer than $e_{4}$ to the root of $T$, then the path through $e_{4}$ and $e_{1}$ passes through $e_{2}$ as well, and we are done; if $e_{4}$ is closer than $e_{2}$ to the root, then the path that we found for $e_{2}$ at the previous reasoning for $C_{2}$ (which has an edge $e'_{1}\in C_{1}$ of degree at least as large as $e_{2}$) passes through $e_{4}$ as well, and we are done again. Hence, Proposition~\ref{pr:treegroup}\eqref{pr:treegroupboundc3} implies that $m(C_{4})\leq\sqrt[3]{n}$ (and $m(C_{5}),m(C_{6}),m(C_{7})$ as well).

We can work analogously by induction for all the thick $C_{i}$; start with $C_{2^{j}}$, and say that at every step $j'<j$ we have that $C_{2^{j'}}$ is the cut lying in the same part of $T$ as $C_{2^{j}}$ with respect to the subdivision of $T$ yielded by the set of all the $C_{i}$ with $i<2^{j'}$ (we can rename $C_{i}$ for $2^{j'}\leq i<2^{j'+1}$ as we please, so there is no loss of generality here). If $j'$ is maximal with respect to the property of having an edge $e_{2^{j'}}$ farther than $e_{2^{j}}$ from the root, we take the path found for $e_{2^{j'}}$ in the case $j'$ (which takes care of all edges for $j''\leq j'$), and then all the $j''$ with $j'<j''<j$ have edges lying on that same path as well, just by being closer than $e_{2^{j}}$ to the root; as before, all these cuts really pass through edges and not vertices, or else $e_{2^{j}}$ would have belonged instead to one of the $C_{2^{j'}}$ with $j'<j$.

Moreover, we repeat the same reasoning for any thin cut, treating it as if it were a thick cut at the step $r+1$ (disregarding all the other thin cuts). Therefore, we have in the end the following bounds:
\begin{align}\label{eq:mcbound}
m(C_{1}) & \leq\frac{2}{3}n, & & \nonumber \\
m(C_{i}) & \leq n^{\frac{1}{j+1}} & & \text{for all }1\leq j<r,2^{j}\leq i<2^{j+1}, \\
m(C) & \leq n^{\frac{1}{r+1}} & & \text{for all }C\text{ thin}. \nonumber
\end{align}

Finally, Proposition~\ref{pr:treegroup}\eqref{pr:treegroupboundc3} implies a bound on the number of thin cuts as well. If every path has the ($\mathcal{C}$3) edges satisfy such a relation, the number of ($\mathcal{C}$3) edges themselves on the path is bounded by $\frac{\log n}{\log 5}$; in the worst case, every two thick cuts $C_{i},C_{j}$ have at least one path whose ($\mathcal{C}$3) edges all lie between them, apart from the two ($\mathcal{C}$3) edges that belong already to $C_{i},C_{j}$. On the other hand, the number of thin cuts between two thick cuts is by construction the same as the maximal number of ($\mathcal{C}$3) edges on a single path between them: thus, there are at most $\frac{2^{r}\log n}{\log 5}$ thin cuts.

After that, we move to the ($\mathcal{C}$1)-cuts. Take any part of $T$ between two consecutive ($\mathcal{C}$3)-cuts: we construct ($\mathcal{C}$1)-cuts on it in the same way as we constructed thin cuts before, i.e.\ constructing the first ($\mathcal{C}$1)-cut by taking every time the first ($\mathcal{C}$1) edge and discarding all the paths passing through it, then the second ($\mathcal{C}$1)-cut by doing the same with the edges left out from the first one, and repeating until all ($\mathcal{C}$1) edges have been taken. We can again bound the number of total ($\mathcal{C}$1)-cuts: by Proposition~\ref{pr:treegroup}\eqref{pr:treegroupboundc1}, the number of ($\mathcal{C}$1) edges on a path is bounded by $O(\log^{2}n)$, so that reasoning as before there will be at most $O(2^{r}\log^{3}n)$ ($\mathcal{C}$1)-cuts in the whole tree.

Finally, we move to the ($\mathcal{C}$2) edges: by how we constructed them, a ($\mathcal{C}$2) edge cannot be followed by another ($\mathcal{C}$2) edge, so for every two consecutive cuts among the ($\mathcal{C}$1)-cuts and ($\mathcal{C}$3)-cuts already defined we simply take the unique ($\mathcal{C}$2)-cut that we are allowed to have between them (if any).

At this point, we have defined on $T$ a set of horizontal cuts that are pairwise non-crossing and such that every edge is contained in a unique cut. What we do now is start from the root and descend the tree, bounding the diameter one cut at a time by some factor. Between any two consecutive horizontal cuts among those we have defined, there is a unique horizontal section: at every step we suppose that we have already bounded $\mathrm{diam}(G)$ by some factor times $\mathrm{diam}(H)$, where $H$ is a subgroup of the product of all the $G_{i}$ in a given horizontal section, and we prove that we can move to the next section at the cost of a new factor; at the end, we will then bound all the factors we have collected. The base case, obviously, is the section made of the sole root, with the tautological bound $\mathrm{diam}(G)\leq 1\cdot\mathrm{diam}(G)$.

Say that at the horizontal section $\{(G_{i},\Omega_{i})\}_{i\in I}$ we have already shown the bound $\mathrm{diam}(G)\leq C\cdot\mathrm{diam}(H)$, for some $H\leq\prod_{i}G_{i}$ and some $C>0$; call $\{(G'_{i},\Omega'_{i})\}_{i\in I'}$ the next horizontal section, and observe that $|I|,|I'|\leq n$ since the $\Omega_{i}$ and the $\Omega'_{i}$ both form partitions of $\Omega$. If the next horizontal cut is a ($\mathcal{C}$1)-cut, we have $I=I'$ and each $G'_{i}$ is a subgroup of $G_{i}$ with $[G_{i}:G'_{i}]\leq n^{C_{3}\log^{5}n}$ for some $C_{3}>0$ by Proposition~\ref{pr:treegroup}\eqref{pr:treegroupboundc1}; calling $H'=H\cap\prod_{i}G'_{i}$ and using Conjecture~\ref{cj:smallindex},
\begin{equation}\label{eq:c1cut}
\mathrm{diam}(H)\leq C_{1}n^{C_{2}}\cdot n^{C_{3}\log^{5}n}\cdot\mathrm{diam}(H'),
\end{equation}
so that we have a bound in terms of the next horizontal section with the extra factor $C_{1}n^{C_{2}+C_{3}\log^{5}n}$ besides $C$. If the next cut is a ($\mathcal{C}$3)-cut (call it $K$), then $I=I'$ again and $G_{i}/G'_{i}$ is either an alternating group or the trivial group; we call $H'=H\cap\prod_{i}G'_{i}$ as before, and we use Lemma~\ref{le:schreierindex}, Proposition~\ref{pr:subdiam} and Proposition~\ref{pr:proddiam} to get
\begin{align}\label{eq:c3cut}
\mathrm{diam}(H) & \leq 4\mathrm{diam}(H,H')\mathrm{diam}(H') \nonumber \\
 & \leq 4\mathrm{diam}\left(\prod_{i}(G_{i}/G'_{i})\right)\mathrm{diam}(H') \nonumber \\
 & <4\cdot\frac{196}{243}\cdot 5n^{4}\mathrm{diam}(\mathrm{Alt}(m(K)))\cdot\mathrm{diam}(H'),
\end{align}
thus giving a new bound with an extra factor of $17n^{4}\mathrm{diam}(\mathrm{Alt}(m(K)))$, say. If the next cut is a ($\mathcal{C}$2)-cut, then for every $(G_{i},\Omega_{i})$ there is a subset $\{(G'_{ij},\Omega'_{ij})\}_{j\in J(i)}$ of the next section with $G_{i}|_{\Omega'_{ij}}=G'_{ij}$ for all $j$ and $\bigcup_{j}\Omega'_{ij}=\Omega$: in that case $G_{i}\leq\prod_{j}G'_{ij}$ in the obvious way, and we only need to reembed $H$ appropriately so as to make it into a subgroup of $\prod_{i,j}G'_{ij}$; we have passed to the next horizontal section without changing the bound, since $H$ and its diameter have remained the same.

Combining \eqref{eq:c1cut} and \eqref{eq:c3cut} with the bound on the number of thick, thin and ($\mathcal{C}$1)-cuts, and recalling that all the leaves are trivial (so that the subgroup $H$ at the last step must be $\{e\}$, and $\mathrm{diam}(H)=1$ by our notational convention of Lemma~\ref{le:schreierindex} and Conjecture~\ref{cj:smallindex}), we obtain
\begin{align}\label{eq:cutsbound}
\mathrm{diam}(G) & \leq(C_{1}n^{C_{2}+C_{3}\log^{5}n})^{O(2^{r}\log^{3}n)}\cdot(17n^{4})^{2^{r}\log n+2^{r}-1}\cdot\prod_{K}\mathrm{diam}(\mathrm{Alt}(m(K))) \nonumber \\
 & =n^{C_{4}2^{r}\log^{8}n}\cdot\prod_{K}\mathrm{diam}(\mathrm{Alt}(m(K)))
\end{align}
for any $G\leq\mathrm{Sym}(n)$ transitive not giant, where $C_{4}$ is some absolute constant and the product is on all ($\mathcal{C}$3)-cuts $K$. This will play the same role as \eqref{eq:realbs92}, which is \cite[Prop.~4.15]{Hel18}: as said before, the essential weakening is that we lost the stronger bound on the indices of the alternating groups in the product.

From here, we proceed along the lines of \cite[\S 6]{Hel18}: we will not go over the details, except for the calculations that differ from the original route. Assume as inductive hypothesis that we have proved Theorem~\ref{th:cfsgfree} for all $n'\leq e^{-\frac{1}{10}}n$ and all $G'\leq\mathrm{Sym}(n')$ transitive. Let $G\leq\mathrm{Sym}(n)$ transitive: if $G$ is not a giant we have \eqref{eq:cutsbound}, while if $G=\mathrm{Sym}(n),\mathrm{Alt}(n)$ we have
\begin{equation}\label{eq:hel1814}
\mathrm{diam}(G)\leq e^{C(\log n)^{3}(\log\log n)^{2}}\mathrm{diam}(G')
\end{equation}
for some $C>0$, where either $G'=\mathrm{Sym}(n'),\mathrm{Alt}(n')$ with $n'\leq e^{-\frac{1}{10}}n$ or $G'\leq\mathrm{Sym}(n)$ is transitive not giant; this is a consequence of Theorem~\ref{th:hel18prod}, which does not use CFSG. In the first case we are done by induction, since
\begin{equation*}
C(\log n)^{3}(\log\log n)^{2}+e^{\frac{1}{\log 2}(\log\log n')^{2}}<e^{\frac{1}{\log 2}(\log\log n)^{2}}
\end{equation*}
for $n$ large; in the second case, we use \eqref{eq:cutsbound} on $G'$ and absorb the factor on the RHS of \eqref{eq:hel1814}, so that we obtain the same bound as in \eqref{eq:cutsbound} even for $G$ giant, with $C_{4}+1$ instead of $C_{4}$ (as long as $n$ is large enough) and where $K$ are ($\mathcal{C}$3)-cuts on the tree of a different transitive group $G'$ (with the same degree, though). Thus, we only have to see whether the bound in \eqref{eq:cutsbound} is enough to imply the statement of the theorem.

Recall \eqref{eq:mcbound}: we can use the inductive hypothesis on each of the diameters in the product of \eqref{eq:cutsbound} since $e^{-\frac{1}{10}}n$ is larger than all the $m(K)$, and therefore we have
\begin{align*}
\log\mathrm{diam}(G) \leq & \ (C_{4}+1)2^{r}\log^{9}n+e^{\frac{1}{\log 2}\left(\log\log\left(\frac{2}{3}n\right)\right)^{2}}+\sum_{j=1}^{r-1}2^{j}e^{\frac{1}{\log 2}\left(\log\log\left(n^{1/(j+1)}\right)\right)^{2}} \\
 & \ +\frac{2^{r}\log n}{\log 5}e^{\frac{1}{\log 2}\left(\log\log\left(n^{1/(r+1)}\right)\right)^{2}}.
\end{align*}
The largest term on the RHS is the second one, which we can bound from above for $n$ large as
\begin{align*}
e^{\frac{1}{\log 2}\left(\log\log\left(\frac{2}{3}n\right)\right)^{2}} & =e^{\frac{1}{\log 2}\log^{2}\left(\log n-\log\frac{3}{2}\right)}\leq e^{\frac{1}{\log 2}\left(\log\log n-\frac{\log 3/2}{\log n}\right)^{2}} \\
 & \leq e^{\frac{1}{\log 2}(\log\log n)^{2}}e^{-\frac{\log 3/2}{\log 2}\frac{\log\log n}{\log n}} \\
 & \leq e^{\frac{1}{\log 2}(\log\log n)^{2}}-\frac{\log 3/2}{2\log 2}\frac{\log\log n}{\log n}e^{\frac{1}{\log 2}(\log\log n)^{2}}.
\end{align*}
The last term depends on our choice of $r$. We choose $r=3$, and for $n$ large we get
\begin{align*}
\frac{2^{r}\log n}{\log 5}e^{\frac{1}{\log 2}\left(\log\log\left(n^{1/(r+1)}\right)\right)^{2}} & =\frac{8\log n}{\log 5}e^{\frac{1}{\log 2}(\log\log n-\log 4)^{2}} \\
 & \leq\frac{8\log n}{\log 5}e^{\frac{1}{\log 2}(\log\log n)^{2}}e^{-\frac{\log 4}{\log 2}\log\log n} \\
 & \leq\frac{8/\log 5}{\log n}e^{\frac{1}{\log 2}(\log\log n)^{2}}.
\end{align*}
The first term is bounded by a constant times $\log^{9}n$. Finally, for the sum we can obtain for $n$ large
\begin{equation*}
\sum_{j=1}^{r-1}2^{j}e^{\frac{1}{\log 2}\left(\log\log\left(n^{1/(j+1)}\right)\right)^{2}}\leq 6e^{\frac{1}{\log 2}(\log\log n-\log 2)^{2}}\leq\frac{6}{\log n}e^{\frac{1}{\log 2}(\log\log n)^{2}}.
\end{equation*}
Combining all of the bounds, we obtain the result.
\end{proof}

\section{Concluding remarks}

It is easy to see that, as long as no deeper analysis is conducted on what possibilities arise for a tree like the one in Proposition~\ref{pr:treegroup}, one cannot even prove Theorem~\ref{th:cfsgfree} with a small improvement as putting $(\log\log n)^{2-\varepsilon}$ in the double exponent.

For instance, there could be a permutation subgroup $G\leq\mathrm{Sym}(n)$ and two disjoint sets $\Omega_{1},\Omega_{2}\subseteq[n]$ with $|\Omega_{1}|=|\Omega_{2}|=\frac{n}{10}$ such that in the tree relative to $G$ the two vertices $(G_{1},\Omega_{1}),(G_{2},\Omega_{2})$ appear with $G_{1}=\mathrm{Alt}\left(\frac{n}{10}\right)$ and $G_{2}=\mathrm{Alt}\left(\sqrt{\frac{n}{10}}\right)\wr\mathrm{Sym}(2)$. In that situation, after $(G_{1},\Omega_{1})$ there is forcibly a ($\mathcal{C}$3) edge where the quotient is the whole $\mathrm{Alt}\left(\frac{n}{10}\right)$, while after $(G_{2},\Omega_{2})$ there is a ($\mathcal{C}$3) edge with quotient $\mathrm{Alt}\left(\sqrt{\frac{n}{10}}\right)$, followed by ($\mathcal{C}$2) edges and then by other ($\mathcal{C}$3) edges again with quotients $\mathrm{Alt}\left(\sqrt{\frac{n}{10}}\right)$: if $n$ is large enough, all other routes in the proof of \cite[Thm.~3.1]{Don18b} in fact cannot occur. In that case, it is not possible to give diameter bounds without having at least to treat, in two separate instances, both $\mathrm{diam}\left(\mathrm{Alt}\left(\frac{n}{10}\right)\right)$ and $\mathrm{diam}\left(\mathrm{Alt}\left(\sqrt{\frac{n}{10}}\right)\right)$; hence, if we assume that we have bounds with $2-\varepsilon$ instead of $2$ for those two factors, the recursion process does not work since for any $C_{1},C_{2}>0$ we have
\begin{align*}
\mathrm{diam}(G) & \geq C_{1}e^{C_{2}(\log\log\frac{n}{10})^{2-\varepsilon}}+C_{1}e^{C_{2}(\log\log\sqrt{\frac{n}{10}})^{2-\varepsilon}} \\
 & >C_{1}e^{C_{2}(\log\log n)^{2-\varepsilon}}(e^{-C_{2}'\frac{(\log\log n)^{1-\varepsilon}}{\log n}}+e^{-C_{2}''(\log\log n)^{1-\varepsilon}}) \\
 & >C_{1}e^{C_{2}(\log\log n)^{2-\varepsilon}}\left(\frac{1}{(\log n)^{\frac{C_{2}'}{(\log\log n)^{\varepsilon}}}}+1-C_{2}''\frac{(\log\log n)^{1-\varepsilon}}{\log n}\right),
\end{align*}
and the last expression in parenthesis is $>1$ for any $\varepsilon>0$, provided that we choose $n$ large enough.

The author, as a matter of fact, believes that such vertices cannot occur in the tree for any $G$: after all, better bounds that use CFSG exist, at least for $\mathrm{Alt}(n)$. An analysis of which possibilities are excluded from the trees in Proposition~\ref{pr:treegroup} is therefore in line with both proving Conjecture~\ref{cj:smallindex} and improving the overall bound in Theorem~\ref{th:cfsgfree}. In fact, that would be the most likely route towards proving the conjecture: a first step involving a description of which groups can (or cannot) appear in such trees, and a second step that proves the conjecture only for those ones that may actually show up. One could even weaken one of the two steps, investigating a larger class of groups or showing a weaker bound for them, and Theorem~\ref{th:cfsgfree} would still work, albeit with less strong bounds (but not necessarily so: we have already observed that there is margin for weakening the conjecture without affecting the final result). Conjecture~\ref{cj:smallindex} in this sense shows a discrete deal of flexibility, whether the reader deems it to be a virtue or a defect.

\section*{Acknowledgements}

The author thanks H.\ A.\ Helfgott for introducing him both to the string isomorphism problem and to the problem of the diameter of permutation groups, and for discussions about his papers \cite{Hel19} \cite{Hel18} on these subjects.

The present paper is part of the author's doctoral thesis \cite{Don20}. The author thanks H.\ A.\ Helfgott, L.\ Bartholdi and P.\ Varj\'u for observations on this particular chapter of his thesis.

\bibliography{Bibliography}
\bibliographystyle{alpha}

\end{document}